\documentclass[12pt]{amsart}
\usepackage{Baskervaldx}
\usepackage[baskervaldx,cmintegrals,bigdelims,vvarbb]{newtxmath}
\usepackage[cal=cm]{mathalfa}
\usepackage{amssymb, enumitem, pgfplots, graphicx, subcaption, amsmath, array, pdfsync, tikz, tikz-cd, color, url, float, fullpage, wrapfig}

\usepackage[font=footnotesize]{caption}

\usetikzlibrary{calc}

\title{Maximal Chow constant and cohomologically constant fibrations\vspace{-2ex}}
\author{Kristin DeVleming and David Stapleton\vspace{-2ex}}

\usepackage[pagebackref]{hyperref}
\hypersetup{
  colorlinks= true,%Colors links instead of ugly boxes
  linkcolor=[RGB]{0,89,42},
  urlcolor=[RGB]{0,89,42},
  citecolor=[RGB]{0,89,42}
}

\definecolor{mycolor}{RGB}{146, 214, 203}
\definecolor{myothercolor}{RGB}{179, 215, 232}

\newtheorem{theorem}{Theorem}
\numberwithin{theorem}{section}
\newtheorem{corollary}[theorem]{Corollary}
\newtheorem{lemma}[theorem]{Lemma}
\newtheorem{proposition}[theorem]{Proposition}

\newtheorem{Lthm}{Theorem}
\newtheorem{Lcon}[Lthm]{Conjecture}
\newtheorem{Lprop}[Lthm]{Proposition}

\theoremstyle{definition}

\newcommand{\hr}[2]{\hyperref[#1]{#2}}

\theoremstyle{definition}
\newtheorem{remark}[theorem]{Remark}
\newtheorem{construction}[theorem]{Construction}

\newtheorem{problem}[theorem]{Problem}
\newtheorem{example}[theorem]{Example}
\newtheorem{definition}[theorem]{Definition}

\def\ZZ{{\mathbb Z}}
\def\QQ{{\mathbb Q}}

\def\AA{{\mathbb A}}
\def\NN{{\mathbb N}}
\def\CC{{\mathbb C}}

\def\PP{{\mathbf{P}}}

\def\Rb{{\mathbf{R}}}

\def\Oc{{\mathcal O}}

\def\Vc{{\mathcal V}}

\def\Dc{{\mathcal D}}

\def\Vbar{{\overline{V}}}
\def\Xbar{{\overline{X}}}
\def\Ybar{{\overline{Y}}}
\def\fbar{{\overline{f}}}

\def\hbar{{\overline{h}}}

\def\Sbar{\overline{S}}

\def\alb{{\mathrm{alb}}}

\def\deg{{\mathrm{deg}}}

\def\rat{{\mathrm{rat}}}
\def\CH{{\mathrm{CH}}}
\def\Alb{{\mathrm{Alb}}}
\def\Spec{{\mathrm{Spec}}}

\def\Pic{{\mathrm{Pic}}}
\def\Hilb{{\mathrm{Hilb}}}
\def\Var{{\mathrm{Var}}}

\def\tors{{\mathrm{tors}}}
\def\Chow{{\mathrm{Chow}}}
\def\bir{{\mathrm{bir}}}
\def\hpo{{H^{p,0}}}
\def\honeo{{H^{1,0}}}
\def\sym{{\mathrm{Sym}}}

\def\Fbar{{\overline{F}}}
\def\DDom{{\Dc\mathrm{Dom}}}

\def\dra{{\dashrightarrow}}
\def\ra{{\rightarrow}}

\def\cl{{\colon}}

\def\Roi{{{Ro\u{\i}}}}

\begin{document}
\maketitle

\section*{Introduction}
\thispagestyle{empty}
Motivated by the study of maximal rationally connected fibrations, introduced by Koll\'ar, Miyaoka, and Mori in \cite{KollarMiyaokaMori} and Campana in \cite{Campana1,Campana2}, we study different notions of fibrations where instead of requiring that the general fibers be rationally connected, we require different types of birational simplicity. The birational invariants we consider are the Chow groups of 0-cycles and the groups of holomorphic $p$-forms. The main result of this paper is the construction of maximal Chow constant and cohomologically constant fibrations.

Consider a fibration (Def.~\ref{fibration}) of smooth complex projective varieties
\[
f\cl X\ra Y.
\]
We say $f$ is a \textbf{Chow constant fibration} if $f_*\cl \CH_0(X)\ra \CH_0(Y)$ is an isomorphism. We say a fibration is \textbf{cohomologically constant} if $f^*\cl H^{p,0}(Y)\cong H^{p,0}(X)$ is an isomorphism for all $p$. These definitions extend to the case when $f$ is a rational map.

\begin{Lthm}\label{MaxConstFibs}
Any smooth complex projective variety $X$ admits a maximal Chow constant (resp. cohomologically constant) fibration:
\[
\eta\cl X\dra Y.
\]
$Y$ is defined up to birational isomorphism and satisfies the following universal property: another fibration $\phi\cl X\dra Z$ is Chow constant (resp. cohomologically constant) $\iff \eta$ factors through $\phi$. Moreover, $\eta$ is almost holomorphic (i.e. there is a nonempty open set $U\subset Y$ over which $\eta$ is proper).
\end{Lthm}

Chow groups of 0-cycles have played an important role in algebraic geometry. Already in the 60s, Mumford \cite{Mum0} observed that even in dimension two, $\CH_0(X)$ can be quite exotic, and proved that if $X$ is a complex K3 surface, then $\CH_0(X)$ is infinite dimensional in a precise sense. \Roi tman proved \cite{RoitmanTorsion} the torsion in $\CH_0(X)$ is isomorphic to the torsion subgroup of $\Alb(X)$. Colliot-Th\'el\`ene, Voisin, and others (see e.g. \cite{VoiCol}) have made major progress in understanding rationality questions by considering specializations of ``universally $\CH_0$-trivial varieties", or equivalently varieties which admit integral decompositions of their diagonals. Beauville and Voisin \cite{BeaVoi} showed that given a K3 surface, there is a distinguished degree 1 cycle $c_X\in \CH_0(X)$ such that many geometrically defined 0-cycles are a multiple of $c_X$. There has been some work in understanding a similar picture for higher dimensional hyperK\"ahler manifolds. Huybrechts \cite{HuyChow} has initiated a study of ``Chow constant subvarieties", i.e. subvarieties $V\subset X$ such that the image of $\CH_0(V)$ in $\CH_0(X)$ is isomorphic to $\ZZ$. Vial \cite{Vial} has studied fibrations similar to the ones considered here, especially from the motivic perspective.

On the other hand, the vector spaces $H^{p,0}(X)$ of holomorphic $p$-forms are some of a variety's most useful birational invariants. From the perspective of this paper, the motivation for considering $p$-forms along with $0$-cycles is Bloch's conjecture.

\begin{Lcon}[Bloch's Conjecture]\label{Bloch} If $X$ is a smooth projective complex variety, then $\CH_0(X)=\ZZ\iff H^{p,0}(X)=0$ for all $p>0$.
\end{Lcon}
\noindent The forward implication is known, but the opposite is known in very few examples (for surfaces with $\kappa(X) < 2$ and a few classes of general type surfaces). There are several generalizations of Bloch's conjecture. For this paper the most relevant generalization is
\begin{Lcon}[{see \cite[Conj. 1.11]{VoisinChowBook}}]\label{GenBloch}
If $H^{p,0}(X)=0$ for all $p>m$ then $\CH_0(X)$ is supported on an $m$-dimensional algebraic subset $V\subset X$, i.e. $\CH_0(V)$ surjects onto $\CH_0(X)$.
\end{Lcon}

\noindent The following proposition explains the relationship between Chow constant and cohomologically constant fibrations and the significance of Conjecture~\ref{GenBloch} to our setting.

\begin{Lprop}
Let $X$ be a smooth complex projective variety and let $Y$ be the base of its maximal Chow constant fibration.
\begin{enumerate}
\item Every Chow constant fibration of $X$ is cohomologically constant.
\item The dimension of $Y$ equals the minimum dimension of an algebraic subset $V\subset X$ such that $\CH_0(X)$ is supported on $V$.
\item If Conjecture~\ref{GenBloch} is true then
\[
\dim(Y)=\max\{p|H^{p,0}(X)\ne 0\},
\]
and $Y$ coincides with the maximal cohomologically constant fibration. Thus, conjecturally, a fibration is Chow constant $\iff$ it is cohomologically constant.
\end{enumerate}
\end{Lprop}

We give some examples and applications which arise in the study of these fibrations. First, we show that being a Chow constant fibration has consequences on the Chow group of the generic fibers.

\begin{Lprop}\label{famsosurfaces}
Let $X$ be a smooth projective threefold with a Chow constant fibration over a curve $B$, and let $\xi=\CC(B)$ be the function field of $B$. Then, there is a divisor $D\subset X$ such that $\CH_0(X_\xi)\otimes \QQ$ is supported on $D_\xi$. Thus, $\CH_0(X_\xi)\otimes \QQ$ is finite dimensional in the sense of Mumford.
\end{Lprop}

\noindent We give several examples of K3 surfaces $X_\xi$ over the function field $\xi$ of a complex curve such that $\CH_0(X_\xi)$ is finite dimensional.

We consider two other classes of fibrations, which are defined only by the properties of their fibers. Let
\[
f\cl X\ra Y
\]
be a fibration of smooth projective varieties. We say $f$ is a \textbf{Chow trivial fibration} if, for a general fiber $X_y$, $\CH_0(X_y)\cong \ZZ$. Likewise, we say that $f$ is a \textbf{cohomologically trivial fibration} if $H^{p,0}(X_y)=0$ for all $p>0$.  (We also define these fibrations when $f$ is a rational map.) They also give rise to maximal fibrations.

\begin{Lthm}\label{MaxTrivFibs}
Any smooth complex projective variety $X$ admits a maximal Chow trivial (resp. cohomologically trivial) fibration:
\[
\eta\cl X\dra Y.
\]
$Y$ is defined up to birational isomorphism, and satisfies the following universal property: if a fibration $\phi\cl X\dra Z$ is Chow trivial (resp. cohomologically trivial) then $\eta$ factors through $\phi$. As in Theorem~\ref{MaxConstFibs}, $\eta$ is almost holomorphic.
\end{Lthm}

For a rational fibration $f\cl X\dra Y$ (see Def.~\ref{fibration}), we have the following chain of implications:
\[
\begin{tikzcd}[arrows=Rightarrow]
\left(\begin{array}{c}
f\text{ is a rationally}\\
\text{conn. fibration}
\end{array}\right)\arrow{r}& \left(\begin{array}{c}
f\text{ is a Chow}\\
\text{trivial fibration}
\end{array}\right)\arrow[r,shift left,"\text{Prop.~\ref{chowimpliescohom}}"]\arrow[d,"\text{Cor.~\ref{cor:chowtrivialimplieschowconstant}}"]&\left(\begin{array}{c}
f\text{ is a cohom.}\\
\text{trivial fibration}
\end{array}\right)\arrow[d,"\mathrm{Cor.~\ref{CohTrivConst}}"]\arrow[l,dashed, shift left,"\text{Conj.~\ref{Bloch}}"]\\
&\left(\begin{array}{c}
f\text{ is a Chow}\\
\text{constant fibration}
\end{array}\right)\arrow[r,shift left,"\text{Prop.~\ref{chowimpliescohom}}"]&\left(\begin{array}{c}
f\text{ is a cohom.}\\
\text{constant fibration}
\end{array}\right)\arrow[l,dashed,shift left,"\text{Conj.~\ref{GenBloch}}"].
\end{tikzcd}
\]

As another application, we note that the study of cohomologically trivial fibrations is relevant to the study of rational singularities. Let $X$ be a variety and $X^{\mathrm{rat}}\subset X$ the locus where $X$ has rational singularities. Koll\'ar has asked the following question: does there exist a partial resolution $\mu\cl X'\ra X$ of $X$ such that $X'$ has rational singularities and $\mu$ is an isomorphism on the preimage of $X^{\mathrm{rat}}$? Motivated by this question, we consider a refinement of the problem in the case of cones. When $X$ is smooth and projective and $L$ an ample line bundle on $X$, then the projective cone $C(X,L)$ has a canonical resolution
\[
\mu\cl \PP(\Oc\oplus L)\ra C(X,L)
\]
by blowing up the cone point. Say a birational model $\Rb$ of $C(X,L)$ is an \textbf{intermediate rationalization of singularities of $C(X,L)$} if $\Rb$ has rational singularities and $\mu$ factors as 
\[
\begin{tikzcd}
\PP(\Oc\oplus L)\arrow[r]\arrow[rr,bend right=40,"\mu"]\arrow[r]&\Rb\arrow[r]&C(X,L).
\end{tikzcd}
\]
We have the following characterization of intermediate rationalizations of singularities of $C(X,L)$ (generalizing the criterion for cones to have rational singularities in \cite[Prop. 3.13]{KollarSMM}).

\begin{Lthm}\label{intratsings}
If $L$ is sufficiently positive, there is a bijective correspondence
\[
\left\{ \begin{array}{l}
\text{intermediate rationalizations of }\\
\text{singularities of $C(X,L)$}
\end{array} \right\}\longleftrightarrow \left\{\begin{array}{l}
\text{regular cohom. trivial fibrations $f\cl X\ra Y$}\\
\text{such that $Y$ has rational singularities}
\end{array}  \right\}.
\]
\end{Lthm}

\noindent One can remove the assumption about the positivity of $L$ by modifying the right hand side.

In \S \ref{sec:cohfibs} we prove some basic facts about cohomologically constant and cohomologically trivial fibrations. We give a criterion for a fibration to be cohomologically constant in terms of a natural distribution/foliation on $X$ (see Def.~\ref{def:VD}) which was suggested to us by Claire Voisin. In \S \ref{sec:chowfibs} we prove analogous facts about Chow constant and Chow trivial fibrations. We show that a fibration is Chow constant if the fibers are Chow constant subvarieties in the sense of Huybrechts (see Theorem~\ref{thm:chowconstantfibers}). We also recall some examples of Chow constant fibrations and prove Proposition~\ref{famsosurfaces}. In \S \ref{sec:ratcones} we prove Theorem~\ref{intratsings}. In \S \ref{sec:maxchowconst} and \S \ref{sec:maxcohomological} we prove Theorem~\ref{MaxConstFibs} and Theorem~\ref{MaxTrivFibs}.  In \S \ref{sec:maxchowconst} we recall the quotient of a variety by an algebraic equivalence relation (which we attribute to \Roi tman). In \S \ref{sec:maxcohomological} we prove that one can produce maximal quotients with fibers in an arbitrary foliation, as suggested to us by Claire Voisin. Lastly, in Appendix \ref{appendix} we prove an elementary result: when $\CH_0(X)$ of a variety $X$ over an arbitrary field $k$ is supported on a curve, then it is finite dimensional in the sense of Mumford.

Unless explicitly stated, we work over $\CC$. All our varieties are by assumption irreducible. By a regular fibration we mean a fibration which is everywhere defined. By abuse of notation if $k\subset \xi$ is a field extension and $X$ is a variety over $k$ then we use $X_\xi$ to denote the base change $X_\xi:=X\times_{\Spec(k)}\Spec(\xi)$.

We would like to thank Ed Dewey, Lawrence Ein, Laure Flapan, Charles Godfrey, Elham Izadi, Robert Lazarsfeld, Stefan Kebekus, Daniel Litt, James M\textsuperscript{c}Kernan, Mircea Musta\c{t}\u{a}, John Ottem, Alex Perry, Ari Shnidman, Fumiaki Suzuki, Burt Totaro, and Claire Voisin for interesting discussions and helpful comments.

\section{Cohomologically Constant and Trivial Fibrations}\label{sec:cohfibs}

In this section, we define cohomologically constant and cohomologically trivial fibrations. We are grateful to Claire Voisin who suggested we define a natural integrable distribution on a variety $\Vc_X$ which controls when a fibration is cohomologically constant. The existence of this distribution is what allows us in \S \ref{sec:maxcohomological} to define the maximal cohomologically constant and cohomologically trivial fibrations.

Let $X$ and $Y$ be projective varieties. Let $f\cl X\dra Y$ be a rational map.

\begin{definition}\label{fibration}
We say $f$ is a \textbf{fibration} if $f$ is dominant and the closure of a general fiber of $f$ is irreducible.
\end{definition}

Recall the following fact about global $p$-forms. 

\begin{lemma}[{\cite[Lem. 2.2]{VoisinSurvey}}]\label{lem:pformsbirational}
For any $p \ge 0$, the group $\hpo(X)$ is a birational invariant among smooth projective varieties.
\end{lemma}

Thus for any rational map $f$ as above we can define a pull-back on $p$-forms by first resolving the rational map $f$
\[
\begin{tikzcd}
&\Xbar\arrow[dl]\arrow[dr,"\fbar"]&\\
X\arrow[rr,dashed,swap,"f"]&&Y
\end{tikzcd}
\]
and defining $f^*$ to be the composition:
\[
f^*\cl \hpo(Y)\xrightarrow{\fbar^*}\hpo(\Xbar)\cong \hpo(X).
\]

\begin{definition}
We say a fibration $f \cl X \dra Y$ between smooth projective varieties is a \textbf{cohomologically constant fibration} if $f^*\cl \hpo(Y)\ra \hpo(X)$ is an isomorphism for all $p$.
\end{definition}

\begin{example}\label{ex:pencilofK3s}
As pullback on $p$-forms is injective, a simple class of examples of cohomologically constant fibrations are those where the domain satisfies $\hpo(X)= 0$ for all $p > 0$. For instance, if $f\cl \PP^3 \dra \PP^1$ is a pencil of quartics, then $f$ is a cohomologically constant fibration.
\end{example}

\begin{remark}
If $X$ is smooth, projective of dimension $n$, and $H^{n,0}(X)\ne 0$, then every cohomologically constant fibration is birational.
\end{remark}

The property of being a cohomologically constant fibration is controlled by a natural distribution on $X$.

\begin{definition}\label{def:VD}
Let $\Vc_X\subset T_X$ be the subsheaf of $T_X$ defined as follows:
\[
\Vc_X(U):=\left\{v\in T_X(U)\middle|
\begin{array}{l}
\forall p>0\text{, }\forall\omega\in \hpo(X)\text{, the}\\
\text{ contraction }\omega\lrcorner (v|_U)= 0\in \Omega_X^{p-1}(U)
\end{array}
\right\}.
\]
We call $\Vc_X$ \textbf{Voisin's distribution}. It is straightforward to show that $\Vc_X$ is integrable (e.g. by applying the invariant formula for the exterior derivative and using that for any form $\omega \in \hpo(X)$, we have $d\omega =0$). Thus $\Vc_X$ generically defines a foliation on $X$, which we call \textbf{Voisin's foliation}.
\end{definition}

\begin{remark}\label{rem:voisindistaskernel}
For each $p>0$ there is a contraction map
\[
\text{cont}_p\cl T_X \rightarrow  \hpo(X)^*\otimes_\CC\Omega_X^{p-1}.
\]
We could equivalently define $\Vc_X:=\cap_{p>0} \ker(\mathrm{cont}_p).$
\end{remark}

\begin{remark}\label{rmk:relvoisindist}
If $f \cl X \ra S$ is a regular fibration of smooth projective varieties with relative dimension $r$, we can also define a relative version of Voisin's distribution (resp. Voisin's foliation) $\Vc_f\subset T_X$. Let $U\subset S$ be the open set where $f$ is smooth and $X_U:=f^{-1}(U)$. Consider the kernel of the relative contraction map:
\[
T_{X_U/U}\ra \bigoplus\limits_{p=1}^{r} f^*\left(f_*(\wedge^p\Omega_{X_U/U})\right)^*\otimes \wedge^{p-1}\Omega_{X_U/U}.
\]
As a subsheaf of $T_{X_U}$, the kernel can be extended to some subsheaf of $T_X$ and $\Vc_f$ is defined as the saturation of any such extension. Then $\Vc_f$ is an integrable distribution and for a general fiber $X_s$ of $f$, the restriction to $X_s$ is Voisin's distribution on the fiber, i.e. $\Vc_f \vert_{X_s} = \Vc_{X_s}$. 
\end{remark}

Consider the following diagram of smooth projective varieties
\[
\begin{tikzcd}
Z\arrow[r,"\psi"]\arrow[d,"\pi"]&X,\\
Y&
\end{tikzcd}
\]
such that both $\pi$ and $\psi$ are surjective and a general fiber of $\pi$ is irreducible. Voisin's distribution is useful for determining when $p$-forms on $X$ descend to $Y$.

\begin{proposition}\label{descendingForms}
The following are equivalent:
\begin{enumerate}
\item Global $p$-forms on $X$ descend to $Y$; i.e. for each $p>0$ and every $p$-form $\omega\in H^{p,0}(X)$ there is a form $\eta\in H^{p,0}(Y)$ such that $\psi^*(\omega)=\pi^*(\eta).$
\item The fibers of the family $Z \to Y$ map into Voisin's foliation; i.e. there is a nonempty open set $U\subset Z$ and a factorization:
\[
\begin{tikzcd}
T_{Z/Y}|_{U}\arrow[rr]\arrow[dr,dashed,"\exists"]&&\psi^*(T_X)|_U.\\
&\psi^*(\Vc_X)|_U\arrow[ur]&
\end{tikzcd}
\]
\end{enumerate}
\end{proposition}

\begin{remark}\label{rem:fibersleaves}
This implies that if $f\cl X\dra Y$ is a fibration, then $f$ is cohomologically constant if and only if a general fiber of $f$ is generically contained in a leaf of Voisin's foliation.
\end{remark}

\begin{proof}[Proof of Proposition]
$(1) \implies (2)$: This direction is straightforward. Let $U\subset Z$ be the nonempty open set where $\pi$ and $\psi$ are both smooth. If $v \in T_{Z/Y}(U)$, then
\[
\psi^*(\omega)|_U\lrcorner v=\pi^*(\eta)|_U\lrcorner v=0\in (\wedge^{p-1}\Omega_Z)_z.
\]
It follows that locally  $\psi^*(\omega)|_U \lrcorner v = 0$ for every global $p$-form $\omega$ and thus $d\psi_* v\in \psi^*(\Vc_X)$.

$(2) \implies (1)$: 
Let $z\in U\subset Z$ be a general point. Then $\pi$ is smooth in a neighborhood of $\pi(z)$. There are coordinates
\[
x_1,\dots,x_r,y_1,\dots,y_s\in \Oc_{Z,z}
\]
in the local ring at $z$ such that the $\{y_i\}$ cuts out the fiber of $\pi$ at $z$ and the $\{x_j\}$ gives coordinates on the fiber. Likewise there is a basis for $(\Omega^p_Z)_z$ locally at $z$ given by $p$-wedges of $dx_i$s and $dy_j$s. The assumption in (2) implies that for any dual basis vector $v_j=\frac{\partial}{\partial x_j}\in (T_{Z/B})_z$ we have
\[
\psi^*(\omega)_z \lrcorner v_j = \psi^*(\omega)_z\lrcorner d\psi_*(v_j) = 0.
\]
Thus in the local coordinates:
\[
\psi^*(\omega)=f_1dy_1\wedge\dots +\dots
\]
and all the terms with $dx_j$s vanish.

Let $W\subset \pi(U)\subset Y$ be a nonempty open set over which $\pi$ is smooth and let $Z_W=\pi^{-1}(W)$. It follows that
\[
\psi^*(\omega)|_{Z_W} \in H^0(Z_W,\pi^*(\Omega^p_W)|_{Z_W}),
\]
and thus $\psi^*(\omega)$ descends to a meromorphic $p$-form $\eta$ on $B$. Showing it extends to a global $p$-form is straightforward. Let $Y'\subset Z$ be a multisection of $\pi$ and let $\pi'=\pi|_{Y'}$. Then we have
\[
\eta=\frac{1}{\deg(\pi')}\text{tr}_{\pi'}(\psi^*(\omega)|_{Y'})
\]
as meromorphic forms on $Y$. But the form on the right is a regular $p$-form, so we are done.
\end{proof}

\begin{definition}
Let $f\cl X\dra Y$ be a fibration, let $V$ be the closure of a general fiber, and let $\Vbar$ be a resolution of singularities of $V$. We say $f$ is a \textbf{cohomologically trivial fibration} if $\hpo(\Vbar)=0$ for all $p>0$.
\end{definition}

\begin{example}
Let $X$ be smooth and projective. Then $f\cl X\ra \Spec(\CC)$ is a cohomologically constant fibration $\iff f$ is a cohomologically trivial fibration$\iff h^{p,0}(X)=0$ $\forall p>0$.
\end{example}

\begin{example}
Continuing with Example \ref{ex:pencilofK3s}, we see that not all cohomologically constant fibrations are cohomologically trivial. If smooth, the closure of a fiber of $f: \PP^3\dra \PP^1$ is a quartic K3 surface $\Vbar\subset \PP^3$, thus $H^{2,0}(\Vbar) \ne 0$ for the general fiber $\Vbar$. 
\end{example}

To relate cohomologically constant and trivial fibrations, we recall a theorem of Koll\'ar:

\begin{theorem}[{\cite[Thm. 7.1]{KollarHigherDirect}}]\label{thm:kollar}
Let $\pi\cl X\ra Z$ be a surjective map between projective varieties, $X$ smooth, $Z$ normal. Let $F$ be the geometric generic fiber of $\pi$ and assume that $F$ is connected. The following two statements are equivalent:
\begin{enumerate}
\item $R^p\pi_*\Oc_X=0$ for all $p>0$;
\item $Z$ has rational singularities and $h^p(F,\Oc_F)=0$ for all $p>0$.
\end{enumerate}
\end{theorem}

The following corollary is a straightforward application of Koll\'ar's theorem using that $H^p(F,\Oc_F) \cong \hpo(F) =0$ for all $p>0$. 

\begin{corollary}\label{CohTrivConst}
If $f\cl X\dra Y$ is a cohomologically trivial fibration of smooth projective varieties, then $f$ is cohomologically constant.
\end{corollary}

Moreover we have:

\begin{corollary}\label{cor:compositionofcohfib}
Let $f\cl X \dra Y$ and $g\cl Y \dra Z$ be fibrations.  If $f$ and $g$ are cohomologically constant (resp. trivial) fibrations, then $g\circ f$ is a cohomologically constant (resp. trivial) fibration.
\end{corollary}

\begin{proof}
If $f$ and $g$ are cohomologically constant, certainly $g \circ f$ is cohomologically constant.  Now assume $f$ and $g$ are cohomologically trivial and let $X_z$ (resp. $Y_z$) denote the closure of the fiber of $g\circ f$ (resp. $g$) over a general point $z\in Z$. Let $\Xbar_z$ (resp. $\Ybar_z$) denote a resolution of singularities of $X_z$ (resp. $Y_z$).  Note that generality of the point $z\in Z$ implies that the induced rational map $\Xbar_z\dra \Ybar_z$ is a cohomologically trivial fibration. Thus by Corollary~\ref{CohTrivConst} we have $\hpo(\Ybar_z)=\hpo(\Xbar_z)=0$ for all $p>0$. 
\end{proof}

Finally we prove an auxiliary result, which will eventually show that all of our maximal fibrations are ``generically proper" over their codomain. Let $X$ and $Y$ be smooth projective varieties and
\[
f\cl X\dra Y
\]
be a dominant rational map and $\Gamma_f\subset X\times Y$ the closure of the graph of $f$. Then $f$ can be extended across the locus where the projection $p\cl \Gamma_f\ra X$ is finite. 

\begin{definition}\label{almhom}
With the setup above, we say the \textbf{exceptional locus of $f$} is the locus in $X$ over which $p$ is not finite. We say $f$ is \textbf{almost holomorphic} if the exceptional locus of $f$ does not intersect the closure of a general fiber.
\end{definition}

\begin{lemma}\label{ruledness}
With the setup above, if $f$ is not almost holomorphic, then $Y$ is uniruled.
\end{lemma}

\begin{proof}
As $X$ is smooth, the fibers of $p$ are rationally chain connected subvarieties of $Y$. Therefore, if the closure of a general fiber meets the exceptional locus of $f$, then there is a rational curve through a general point in $Y$.
\end{proof}

\section{Chow constant and Chow trivial fibrations}\label{sec:chowfibs}

In this section we define Chow constant and Chow trivial fibrations. We show that the property of being a Chow constant fibration is equivalent to having fibers which are Chow constant cycles in the sense of Huybrechts \cite[Def. 3.1]{HuyChow}. We give some examples of Chow constant fibrations, focusing for the sake of exposition on Chow constant fibrations where the fibers are K3 surfaces. We also prove Proposition~\ref{famsosurfaces} relating Chow constant fibrations and the Chow groups of their generic fibers. To start, recall the following fact about $\CH_0(X)$.

\begin{lemma}[{\cite[Ex. 16.1.11]{Fulton}}]\label{lem:chow0birational}
The group $\CH_0(X)$ is a birational invariant among smooth projective varieties.
\end{lemma}

Therefore, for a fibration $f$ we may define a pushforward at the level of 0-cycles in analogy with our definition of pull-back of $p$-forms. Let
\[
\begin{tikzcd}
&\Xbar\arrow[dl]\arrow[dr,"\fbar"]&\\
X\arrow[rr,dashed,"f"]&&Y
\end{tikzcd}
\]
be a resolution of the map $f$. Then we define $f_*$ to be the composition:
\[
f_*\cl \CH_0(X)\cong \CH_0(\Xbar)\xrightarrow{\fbar_*}\CH_0(Y).
\]
This is independent of the resolution of $f$.

\begin{definition}
We say that a fibration $f \cl X \dra Y$ between smooth projective varieties is a \textbf{Chow constant fibration} if $f_*$ is an isomorphism.
\end{definition}

It will be useful to consider Chow-theoretic properties of subvarieties. Let $V\subset X$ be a subvariety and let $\Vbar$ be a resolution of singularities of $V$.

\begin{definition}
We say $V$ is a \textbf{Chow constant subvariety} (see \cite[Def. 3.1]{HuyChow}) if for any two points $x_1,x_2\in V$ we have $x_1= x_2\in \CH_0(X)$. We say that $V$ is a \textbf{Chow trivial subvariety} if $\CH_0(\Vbar)\cong\ZZ$.
\end{definition}

\begin{definition}
We say a fibration is a \textbf{Chow trivial fibration} if the closure of a general fiber is a Chow trivial subvariety.
\end{definition}

Now we show that Chow constant fibrations are exactly the fibrations where the general fibers are Chow constant subvarieties.

\begin{theorem}\label{thm:chowconstantfibers}
Let $f\cl X\dra Y$ be a fibration of smooth projective varieties. Then $f$ is a Chow constant fibration $\iff$ a general fiber of $f$ is a Chow constant subvariety.
\end{theorem}

\begin{proof}
If $f_*$ is a Chow constant fibration, then a general fiber is clearly a Chow constant subvariety. For the other direction, we first show that $f_*$ is an isomorphism modulo torsion, i.e. after tensoring with $\QQ$. Then, we use \Roi tman's theorem to complete the proof.

By definition of $f_*$ we are free to resolve $f$, i.e. assume that $f$ is everywhere defined. It is clear that $f_*\cl \CH_0(X)\ra \CH_0(Y)$ is a surjection. We must show it is also injective. Let $i:Z\hookrightarrow X$ be a smooth multisection of $f$ of degree $d$, i.e. a smooth and closed subvariety which maps generically finitely onto $Y$. Let $g=f|_Z$. There is an open set $U\subset Y$ over which $g$ is \'etale such that for any point $y\in U$ the fiber $X_y$ is a Chow-constant subvariety.

Both of the following compositions
\[
(i_*\circ g^*)\circ f_*\cl \CH_0(X)\ra \CH_0(X) \text{ and }f_*\circ (i_*\circ g^*)\cl \CH_0(Y)\ra \CH_0(Y),
\]
are equal to multiplication by $d$. For the second map $f_*\circ (i_*\circ g^*)$ this is straightforward. To prove it for $(i_*\circ g^*)\circ f_*$, we use the following: any $\alpha\in \CH_0(X)$ can be moved so that it is supported on $f^{-1}(U)$, and for any point $x\in f^{-1}(U)$ we have $(i_*\circ g^*)\circ f_*(x)$ is a union of $d$-points in $X_{f(x)}$. As $X_{f(x)}$ is a Chow constant subvariety, we have $(i_*\circ g^*)\circ f_*(x)=d\cdot x\in \CH_0(X)$. Thus $(i_*\circ g^*)\circ f_*$ is equal to multiplication by $d$, which implies
\[
f_*\otimes \QQ \cl \CH_0(X)\otimes \QQ \ra \CH_0(Y)\otimes \QQ
\]
is an isomorphism. Therefore the kernel of $f_*$ is $d$-torsion.

The previous paragraph shows that if $x_1,x_2\in X_y$ are two points in a fiber of $f$ then the difference $x_1-x_2$ is torsion in $\CH_0(X)$. Let
\[
\alb_X\cl X\ra \Alb(X)
\]
be the Albanese map of $X$. For any two points $x_1,x_2\in X_y$, the difference $\alb_X(x_1)-\alb_X(x_2)\in\Alb(X)$ is torsion. But as $X_y$ is connected and the torsion points are countable, this implies that the map $\alb_X$ is constant on the fibers of $f$. So there is a factorization:
\[
\begin{tikzcd}
&Y\arrow[dr,"\exists"]&\\
X\arrow[ur,"f"]\arrow[rr,"\alb_X"]&&\Alb(X).
\end{tikzcd}
\]
Now \Roi tman's theorem \cite{RoitmanTorsion} implies the composition
\[
\CH_0(X)_\tors \xrightarrow{f_*} \CH_0(Y)_\tors \ra \Alb(X)_\tors\cong\CH_0(\Alb(X))_\tors
\]
is an isomorphism. This proves that $f_*$ is injective, so it is an isomorphism.
\end{proof}

\begin{remark}
In the previous theorem one can weaken the smoothness hypotheses quite a bit. To show that the kernel of $f_*$ is a torsion group requires no smoothness. To conclude that the kernel of $f_*$ is trivial, it would suffice to assume that $X$ and $Y$ are normal, and that a resolution of singularities $\Xbar$ of $X$ induces an isomorphism $\CH_0(\Xbar)\cong \CH_0(X).$
\end{remark}

The following corollary is immediate.

\begin{corollary}\label{cor:chowtrivialimplieschowconstant}
Let $f\cl X\dra Y$ be a fibration of smooth projective varieties. If $f$ is a Chow trivial fibration, then it is a Chow constant fibration.
\end{corollary}

Moreover, one may compose these fibrations:

\begin{corollary}\label{cor:compositionofchowfib}
Let $f\cl X\dra Y$ and $g\cl Y\dra Z$ be two fibrations of projective varieties.  If $f$ and $g$ are both Chow constant (resp. trivial) fibrations then $g\circ f$ is a Chow constant (resp. trivial) fibration.
\end{corollary}

\begin{proof}
If $f$ and $g$ are Chow constant, it is straightforward to see $g\circ f$ is Chow constant. Proving triviality follows an similar argument to the proof of Corollary \ref{cor:compositionofcohfib}. We just note that by the previous theorem, if
\[
f\cl X\dra Y
\]
is a Chow trivial fibration over a Chow trivial variety $Y$, then $\CH_0(X)=\ZZ.$  
\end{proof}

\begin{proposition}\label{chowimpliescohom}
Let $f\cl X\dra Y$ be a fibration.
\begin{enumerate}
\item If $f$ is a Chow constant fibration, then $f$ is a cohomologically constant fibration.
\item If $f$ is a Chow trivial fibration, then $f$ is a cohomologically trivial fibration.
\end{enumerate}
\end{proposition}

\begin{proof}
Part (1) holds by the following lemma. Part (2) can be seen as a special case of the following lemma or follows from Mumford's original paper \cite{Mum0}.
\end{proof}

Suppose
\[
\begin{tikzcd}
Z\arrow[r,"\psi"]\arrow[d,"\pi"]&X\\
B&
\end{tikzcd}
\]
is a diagram of smooth projective varieties such that $\pi$ is surjective and a general fiber of $\pi$ is irreducible.

\begin{lemma}\label{lem:chowconstantdescentofforms}
Let $Z_b:=\pi^{-1}(b)$ be a fiber of $\pi$ over a general point $b\in B$. If the image $\psi(Z_b)$ is a Chow constant subvariety, then for any $\omega \in \hpo(X)$ there is an $\eta\in \hpo(B)$ such that $\psi^*(\omega)=\pi^*(\eta)$.  As a special case, this implies that if 
$\CH_0(X) = \mathbb{Z}$, then $H^{p,0}(X) = 0$ for all $p > 0$.  
\end{lemma}

\begin{proof}
This is well known but we include a proof for the convenience of the reader. We follow the outline of the proof \cite[Thm. 3.13]{VoisinChowBook} which is a very similar situation. First we reduce to the case that $\pi$ has a section. Taking a generically finite cover $B'\ra B$ we can assume there is a diagram:
\[
\begin{tikzcd}
Z'\arrow[rr,bend left=40,swap,"\psi'"]\arrow[r]\arrow[d,"\pi'"]&Z\arrow[r,swap,"\psi"]\arrow[d,"\pi"]&X\\
B'\arrow[r,"\phi"]&B&
\end{tikzcd}
\]
satisfying (1) $\pi'$ has a section $\sigma: B'\ra Z'$,
(2) $Z'$ and $B'$ are smooth, projective varieties, and (3) there is a nonempty open set $U\subset B$ over which $\pi'$ is the base change of $\pi$.

Note that $\psi'$ and $\pi'$ satisfy the hypotheses in the lemma. Furthermore, if there exists $\eta'\in \hpo(B')$ such that $\pi'^*(\eta')=\psi'^*(\omega)$ then setting
\[
\eta=\frac{1}{\deg(\phi)}\mathrm{tr}_\phi(\eta')\in\hpo(B)
\]
we have $\pi^*(\eta)=\psi^*(\omega)$.

Thus it suffices to prove the lemma in the case that $\pi$ has a section $\sigma\cl B\ra Z$, which we now assume. Consider the following two cycles in $Z\times X$:
\[
\Gamma_\psi=\{(z,\psi(z))\in Z\times X\}\text{ and }\Gamma_{\psi\circ\sigma\circ \pi}=\{(z,\psi(\sigma(\pi(z))))\in Z\times X\}.
\]
The assumption that $\psi(Z_b)$ is a Chow constant subvariety implies that the fibers $(\Gamma_\psi)_z$ and $(\Gamma_{\psi\circ\sigma\circ \pi})_z$ are rationally equivalent. By Bloch and Srinivas's result \cite[Prop. 1]{BlochSrinivas}, we can write
\[
\Gamma_\psi = \Gamma_{\psi\circ\sigma\circ\pi} + W\in \CH_*(Z\times X)\otimes \QQ
\]
where $W$ is supported on $D\times X$ for some divisor $D\subset X$. As a consequence the map
\[
(\Gamma_\psi)_*=\psi^*\cl \hpo(X)\ra \hpo(Z)
\]
is a sum of the following maps:
\[
\hpo(X)\xrightarrow{(\psi\circ \sigma)^*}\hpo(B)\xrightarrow{\pi^*}\hpo(X)
\]
and
\[
W_*\cl \hpo(X)\ra \hpo(Z).
\]
The second map must vanish as it factors through the Gysin pushforward of the group $H^{p-1,-1}(D)=0$ (see \cite[Thm. 3.13]{VoisinChowBook} for an elaboration on this point). It follows that the pullback $\psi^*(\omega)$ of any $p$-form on $X$ can be written as the $\pi^*(\eta)$ for some $\eta\in \hpo(B).$
\end{proof}

Now, we present several examples of Chow constant fibrations. There are two main sources of examples: fibrations where the domain has $\CH_0(X)=\ZZ$ and examples which arise as quotients by finite group actions. We think the following is a natural problem:

\begin{problem}
Find new techniques for constructing Chow constant fibrations.
\end{problem}

\begin{example}
If $X$ is a rationally connected variety (or any variety with $\CH_0(X)=\ZZ$) then any fibration $f\cl X\dra Y$ is a Chow constant fibration.
\end{example}

Now we recall an example of Bloch, Kas, and Lieberman \cite{BlochKL}. Those authors were interested specifically in the case of surfaces fibered over a curve. We rephrase their construction in the higher dimensional setting.

\begin{example}\label{ex:bloch}
Let $G=\ZZ/d\ZZ$ and let $Y$ be a smooth projective variety with a $G$-action such that the quotient
$$
\tau\cl Y\ra Z:=Y/G
$$
is smooth and satisfies $\CH_0(Z)=\ZZ$. Bloch, Kas, and Lieberman consider the case when $Y$ is a cyclic cover of $Z=\PP^1$. Another example of interest is when $Y$ is a K3 surface which is either a double cover of $\PP^2$ or the double cover of an Enriques surface.

Let $E'$ be an elliptic curve with a choice of $d$-torsion point $\epsilon\in E'$ so that $G$ acts freely on $E'$ by translation by $\epsilon$. Thus $G\times G$ acts on $Y\times E'$ and we can consider the quotient
$$
\sigma\cl Y\times E'\ra X:=(Y\times E')/G
$$
by the diagonal action of $G$ on $Y\times E'$. Define $E:=E'/G$. There is a map
$$
\pi\cl X\ra E.
$$
Note that $\pi$ is an isotrivial family with all fibers being isomorphic to $Y$.
\end{example}

\begin{proposition}\label{BKLprop}
The map $\pi\cl X \ra E$ is a Chow constant fibration.
\end{proposition}

\begin{proof}
This argument is due to Bloch, Kas, and Lieberman (\cite{BlochKL}). First we show that $\pi=\alb_X.$ One can compute:
\[
\honeo(X)\cong \honeo(Y\times E')^G \cong \honeo(Y)^G\oplus \honeo(E')^G\cong \honeo(Z)\oplus \honeo(E).
\]
Now $\honeo(Z)=0$ as $\CH_0(Z)\cong\ZZ$ (by Lemma \ref{lem:chowconstantdescentofforms}). It follows that $\Alb(X)$ is isogenous to $E$. But as $\pi$ has connected fibers we get $\alb_X=\pi$. Thus, as in the proof of Theorem \ref{thm:chowconstantfibers} (i.e. by applying \Roi tman's theorem) it suffices to show that
\[
\pi_*\otimes \QQ \cl \CH_0(X)\otimes\QQ\ra \CH_0(E)\otimes \QQ
\]
is an isomorphism.

Note that there is a $G\cong (G\times G/G)$ action on $X$. Taking the quotient by $G$ we get the following commuting diagram
\[
\begin{tikzcd}
X\arrow[rr,"\pi"]\arrow[dr,swap,"q"]&&E,\\
&Z\times E\arrow[ur,swap,"p"]&
\end{tikzcd}
\]
where $q$ is the quotient map and $p$ is the projection onto $E$. Then we have $\CH_0(Z\times E)\cong\CH_0(E)$ (as $Z$ is a Chow trivial variety), and by averaging: $\CH_0(Z\times E)\otimes \QQ\cong (\CH_0(X)\otimes \QQ)^G$.

So it suffices to show that $\CH_0(X)\otimes \QQ\cong (\CH_0(X)\otimes \QQ)^G$, i.e. we want to show that for any $x\in X$ and any $g\in G$ we have
\[
x=g\cdot x\in \CH_0(X)\otimes \QQ.
\]
As $X=(Y\times E')/G$ there a $G$-equivariant map $E'\ra X$ whose image contains $x\in X$. As the action of $G$ on $E'$ is translation by a $d$-torsion point, the Abel-Jacobi theorem implies that $x=g\cdot x\in \CH_0(E')\otimes \QQ$. Pushing forward to $X$ proves the result.
\end{proof}

\begin{example}
In the above construction we can replace $E'$ with $\PP^1$ and replace translation by a $d$-torsion point with multiplication by a $d$-th root of unity. If we further assume that the quotient $X=(Y\times \PP^1)/G$ is smooth then the same proof as above implies that the map
$$
\pi\cl X\ra \PP^1\cong(\PP^1/G)
$$
is a Chow-constant fibration over $\PP^1$, hence we have $\CH_0(X)\cong\ZZ$. For example, when $Y$ is a K3 surface which double covers an Enriques surface then the quotient $(Y\times \PP^1)/(\ZZ/2\ZZ)$ is smooth so has this property.
\end{example}

Now we prove Proposition~\ref{famsosurfaces}. Suppose that $X$ is a smooth projective threefold, $B$ is a smooth projective curve, and $\pi\cl X\ra B$ is a Chow constant fibration. Let $\xi=\CC(B)$ be the function field of $B$. We show that the property of being a Chow constant fibration has consequences on the group $\CH_0(X_\xi)$. Recall that given a smooth surface $X$ with $h^{2,0}(X)\ne0$ over an uncountable algebraically closedy field of characteristic 0, Mumford showed that $\CH_0(X)$ is not finite dimensional in the following sense:

\begin{definition}\label{MumFinite}
Let $\xi$ be an arbitrary field, let $X$ be a variety over $\xi$ and let $\CH_0(X)_0$ be the 0-cycles of degree $0$. We say $\CH_0(X)$ is \textbf{finite dimensional in the sense of Mumford} if there exists a $d$ such that every 0-cycle of degree 0 is rationally equivalent to a difference of effective $0$-cycle of degree $d$ (i.e. the map of sets:
\begin{center}
$\mathrm{Sym}^d(X)(\xi)\times \mathrm{Sym}^d(X)(\xi)\ra \CH_0(X)_0$\\
$(\sum x_i)\times(\sum y_j) \mapsto (\sum x_i - \sum y_j)$
\end{center}
is surjective). Taking some personal liberties, we say $\CH_0(X)\otimes \QQ$ is \textbf{finite dimensional in the sense of Mumford} if there exists $d>0$ such that the map
\begin{center}
$\mathrm{Sym}^d(X)(\xi)\times \mathrm{Sym}^d(X)(\xi)\times \QQ\ra \CH_0(X)_0\otimes \QQ$\\
$(\sum x_i)\times(\sum y_j)\times \alpha \mapsto (\sum x_i - \sum y_j)\alpha$
\end{center}
is surjective.
\end{definition}

\begin{proof}[Proof of Proposition \ref{famsosurfaces}]
As $\pi\cl X \ra B$ is a Chow constant fibration over a curve, if
$$
i\cl C\hookrightarrow X
$$
is any multisection of $\pi$ (i.e. a curve so that $\pi\circ i\cl C \ra B$ is surjective) then we have $\CH_0(X)$ is supported on $C$. That is, the map
$$
\CH_0(C)\ra \CH_0(X)
$$
is surjective. So we can apply Bloch and Srinivas's result \cite[Prop. 1]{BlochSrinivas} to give a decomposition of the diagonal
\[
\Delta_X = Z_1+Z_2 \in \CH_0(X\times X)\otimes \QQ,
\]
where $Z_1$ is supported on $C\times X$ and $Z_2$ is supported on $X\times D$ for some divisor $D\subset X$.

Now suppose that $\alpha\subset X$ is any irreducible curve, and let $p_1, p_2$ denote projections of $X\times X$ onto each factor. Then we can use the decomposition of diagonal to write
$$
\begin{array}{rcl}
[\alpha]&=&{p_2}_*(p_1^*([\alpha])\cdot \Delta_X)\\
&=&{p_2}_*(p_1^*([\alpha])\cdot(Z_1+Z_2))\in \CH_1(X)\otimes \QQ
\end{array}
$$
By assumption, $Z_1$ is supported on $C\times X$. The pullback of $[\alpha]$ to $\CH_*(C\times X)$ under the composition
$$
C\times X \ra C \xrightarrow{i} X
$$
vanishes as the intersection $[C]\cdot [\alpha]=0$ for dimension reasons (they are both curves in a threefold). Thus we have the intersection $p_1^*([\alpha])\cdot Z_1=0$. Therefore $[\alpha]={p_2}_*(p_1^*[\alpha]\cdot Z_2)$ is supported on $D$, which implies $\CH_1(X)\otimes\QQ$ is supported on the divisor $D$.

So we have shown that the map
$$
\CH_1(D)\otimes \QQ \ra \CH_1(X)\otimes \QQ
$$
is surjective. The localization sequence for Chow groups implies that for any open set $U\subset X$ we have a commutative diagram:
$$
\begin{tikzcd}
\CH_1(D)\otimes\QQ \arrow[r] \arrow[d]& \CH_1(D\cap U)\otimes \QQ\arrow[d]\\
\CH_1(X)\otimes \QQ\arrow[r]&\CH_1(X\cap U)\otimes\QQ
\end{tikzcd}
$$
and moreover, all the maps in the diagrams are surjections.

Finally we use the following expression for $\CH_0(X_\xi)\otimes \QQ$:
$$
\CH_0(X_\xi)\otimes \QQ =\operatorname*{colim}\limits_{\emptyset \ne V\subset B} \left(\CH_1(X\cap \pi^{-1}(V))\otimes \QQ\right),
$$
and likewise
$$
\CH_0(D_\xi)\otimes \QQ =\operatorname*{colim}\limits_{\emptyset \ne V\subset B}\left(\CH_1(D\cap \pi^{-1}(V))\otimes \QQ\right).
$$
(The colimit is taken over nonempty open subsets $V\subset B$.) Thus the map
$$
\CH_0(D_\xi)\otimes \QQ\ra \CH_0(X_\eta)\otimes \QQ
$$
is surjective, i.e. $\CH_0(X_\xi)\otimes \QQ$ is supported on the curve $D_\xi$. By Corollary~\ref{cor:chow0finitedim}, $\CH_0(X_\xi)\otimes \QQ$ is finite dimensional.
\end{proof}

\begin{example}
This gives examples of K3 surfaces $X_\xi$ over function fields of curves such that $\CH_0(X_\xi)$ is finite dimensional. For example, if $Y$ is a K3 surface which double covers $\PP^2$ or an Enriques surface, and we apply the construction of Bloch, Kas, and Lieberman (see Prop.~\ref{BKLprop}) then we get a Chow constant fibration $\pi\cl X\ra B=E.$ (To see other examples where $\CH_0(X)$ is finite dimensional for K3 surfaces over function fields, and related discussion see \cite[\S12.22]{HuyBook}.)
\end{example}

\begin{example}
It is frequently possible to explicitly compute $\CH_0(X_\xi)$. For example, let $X\subset \PP^3\times \PP^1$ be the total space of a pencil of quartics in $\PP^3$ (c.f. Example \ref{ex:pencilofK3s}) and assume that the base locus of the pencil is a smooth complete intersection curve $C$ of type $(4,4)$. Thus $X$ is the blow up of $\PP^3$ at the curve $C$, and the map $p\cl X\ra \PP^1$ is a Chow constant fibration. Let $E\cong C\times \PP^1$ be the exceptional divisor of $q\cl X\ra \PP^3$. Then by \cite[Thm. 2.13]{VoisinSurvey}, we have
\[
\CH_1(X) = \CH_0(C)\oplus \ZZ[\ell]
\]
where $\ell\subset X$ is the preimage of a general line in $\PP^3$. In this family, there is necessarily a fiber $X_0$ such that the quartic surface $X_0$ contains a line (quartics containing a line are an ample divisor in the projective space of quartics). It follows that
\[
\PP^3\setminus X_0 \cong X\setminus (E\cup X_0)
\]
and by the localization sequence we get
\[
\CH_1(X_0)\ra \CH_1(\PP^3)\ra \CH_1(\PP^3\setminus X_0) \cong \CH_1(X\setminus (E\cup X_0))=0.
\]
We also have $E\cap X_0 \cong C$ and $E\setminus C\cong C\times \AA^1$ which gives a diagram of surjections
\[
\begin{tikzcd}
\CH_1(C\times \AA^1)\arrow[r,twoheadrightarrow]\arrow[dr,twoheadrightarrow]& \CH_1(X\setminus X_0)\arrow[d,twoheadrightarrow]\\
&\CH_0(X_\xi)
\end{tikzcd}
\]
Moreover, it is easy to show that in fact $\CH_0(X_\xi)\cong \CH_0(C\times \AA^1)\cong \Pic(C_\xi)\cong \Pic(C).$
\end{example}

\begin{example}
The previous example can be modified to give an example of a K3 surface $X_\xi$ over $\xi =\CC(\PP^1)$ with $\CH_0(X_\xi)\cong \ZZ$. Let $X_0, X_\infty \subset \PP^3$ be two quartics which are both the union of four transverse planes, so that the singularities of $X_0$ and $X_\infty$ do not meet.
\end{example}

\begin{wrapfigure}{r}{-20pt}
\centering
\hspace{5pt}
\begin{tikzpicture}[scale=1]
 
\coordinate (A) at (1.75,.5,0);
\coordinate (B) at (-1.75,.5,0);
\coordinate (C) at (2,0,-3);
\coordinate (D) at (0,4,0);

\coordinate (p1) at (-1,2,0);
\coordinate (p2) at (1,2,0);
\coordinate (p3) at (1,2,-1.5);
\coordinate (p4) at (-.75,2.5,0);
\coordinate (p5) at (.65,.5,0);
\coordinate (p6) at (.75,2.5,0);
\coordinate (p7) at (-.65,.5,0);
\coordinate (p8) at (1.95,0.1,-2.4);
\coordinate (p9) at (1.83,0.34,-.98);
\coordinate (p10) at (.75,2.5,-1.125);
\coordinate (v1) at (-.3,2,.67);
\coordinate (v2) at (1.72,2,-1.1);
\coordinate (i1) at (-.4,2,0);
\coordinate (i2) at (.4,2,0);
\coordinate (i3) at (0,1.4285,0);
\coordinate (j1) at (1,2,-1.1);
\coordinate (j2) at (1,2,-.5);
\coordinate (j3) at (1.29,1.43,-1.07);

\coordinate (newp10) at (.75,2.5,-1.125);

%\draw[-,thick,fill=myothercolor] (A)--(D)--(B)--cycle;
%\draw[-,fill=myothercolor] (A)--(C)--(D)--cycle;
%\draw[-,thick] (A)--(C);
%\draw[-,thick] (D)--(C);
\draw[-,thick,color=myothercolor,fill=myothercolor] (A)--(B)--(D)--(C)--cycle;

\draw[-,thick,color=gray] (A)--(D);

\draw[-,line width=1] (p1)--(p2);
\draw[-,line width=1] (p2)--(p3);
\draw[-,line width=1] (p4)--(i1);
\draw[-,line width=1] (p5)--(i3);
\draw[-,line width=1] (p6)--(p7);
\draw[-,line width=1] (p6)--(p8);
\draw[-,line width=1] (p9)--(j3);
\draw[-,line width=1] (newp10)--(j1);

\draw[fill] (p6) circle [radius=.02];
\draw[fill] (p2) circle [radius=.02];

%\draw[-,thick,color=gray] (A)--(B)--(D)--(C)--cycle;

\draw[fill] (v1) circle [radius=.01];
\draw[fill] (v2) circle [radius=.01];

\draw[-] (i1)--(v1);
\draw[-] (i2)--(v1);
\draw[-] (i3)--(v1);

\draw[-] (j1)--(v2);
\draw[-] (j2)--(v2);
\draw[-] (j3)--(v2);

\draw[fill=mycolor] (i1)--(i2)--(v1);
\draw[fill=mycolor] (i2)--(i3)--(v1);

\draw[fill=mycolor] (j1)--(j2)--(v2);
\draw[fill=mycolor] (j2)--(j3)--(v2);
\end{tikzpicture}
\caption{Some of the lines in the intersection of the degenerate quartic surfaces.}
\end{wrapfigure}
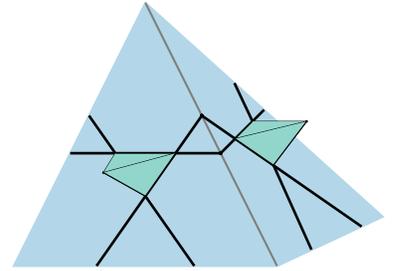
\vspace{-12pt}
Let $X\subset \PP^3\times \PP^1$ be the total space of the pencil spanned by $X_0$ and $X_\infty$ (i.e. the blow up of the intersection $C=X_0\cap X_\infty$). Let $p\cl X\ra \PP^3$ be the blow up map and $E=C\times \PP^1$ be the exceptional divisor. Then, we have $C=\ell_1\cup \dots \cup \ell_{16}$ is the union of 16 lines. The localization sequence again shows that\\
$\hspace{4cm}\CH_1(E)\ra \CH_1(X)$\\
\noindent is surjective, and we can write $\CH_1(E)=\CH_1(C\times \PP^1)\cong \CH_0(C)\oplus \CH_1(C).$ But as $C$ is a connected union of rational curves we have $\CH_0(C)\cong \ZZ$. And the group $\CH_1(C)$ is the kernel of the composition\\
$\hspace{3cm}\CH_1(E)\ra \CH_1(X) \ra \CH_0(X_\xi),$\\
\noindent which shows $\CH_0(X_\xi)\cong\CH_0(C) \ZZ.$\\

\section{Rationalizations of Singularities of Cones}\label{sec:ratcones}

Motivated by Koll\'ar's question, we consider rationalizations of singularities of cones and prove a more general version of Theorem~\ref{intratsings}.

\begin{definition}
Over an algebraically closed field of characteristic zero, a variety $X$ has \textbf{rational singularities} if, for any proper birational morphism $\mu\cl X'\to X$, $R^p\mu_*\Oc_{X'} = 0$ for all $p >0$.
\end{definition}

\noindent Let $X^{\rat}\subset X$ denote the open set where $X$ has rational singularities.

\begin{definition}
We say a proper birational morphism $\mu\cl X'\to X$ is a \textbf{rationalization of singularities} of $X$ if $X'$ has rational singularities.  We say that $\mu$ is a \textbf{strict rationalization of singularities} if $X'$ has rational singularities and $\mu$ gives an isomorphism between $\mu^{-1}(X^{\rat})$ and $X^{\rat}$.
\end{definition}

\noindent Thus Koll\'ar asks whether or not strict rationalizations of singularities exist.

We will study certain rationalizations of singularities of cones. Let $X$ be a smooth variety and let $C(X,L)$ denote the projective cone over $L$ (see \cite[pg. 97]{KollarSMM}). Then $C(X,L)$ has a natural resolution:
\[
\mu\cl \PP(\Oc\oplus L)\ra C(X,L)
\]
given by blowing up the cone point. Thus Koll\'ar's question is trivial for cones (either $C(X,L)$ or $\PP(\Oc\oplus L)$ solves the problem), however following refinement remains interesting:

\begin{problem}
To what extent do there exist minimal rationalizations of singularities?
\end{problem}

\noindent We give a partial answer to this question in the case of cones.

\begin{definition}
We say a birational model $\Rb$ of $C(X,L)$ is an \textbf{intermediate rationalization of singularities of a cone} $C(X,L)$ if $\Rb$ has rational singularities and fits into a diagram
\[
\begin{tikzcd}
\PP(\Oc\oplus L)\arrow[r]\arrow[rr,bend right=40,"\mu"]\arrow[r]&\Rb\arrow[r]&C(X,L)
\end{tikzcd}
\]
\end{definition}

%We will study rationalizations of singularities of cones.  Certainly cones over smooth varieties (or any variety $Z$ with isolated singularities) admit strict rationalizations of singularities: either $Z$ already has rational singularities and $\pi$ is the identity, or there is an isolated singular point that is not rational, in which case we resolve that singularity and $\pi$ is an isomorphism away from that point.   

%However, we can go further and classify all rationalizations of singularities of cones.  Let $X$ be a smooth projective variety and let $L$ be an ample line bundle on $X$.  Following the notation of Koll\'ar in \cite[Chapter 3]{KollarSMM}, we define the \textit{affine cone} over $X$ with conormal bundle $L$ as \[C_a(X,L) = \Spec_k \oplus H^0(X,L^m). \]  Let $C_p(X,L)$ be the corresponding projective cone.  Koll\'ar has a criterion for when this cone has rational singularities. 

\noindent We recall the criterion for cones to have rational singularities. 

\begin{theorem}{\cite[Prop. 3.13]{KollarSMM}}\label{thm:rationalcones}
Let $X$ be a complex projective variety with rational singularities.  Let $L$ be an ample line bundle on $X$.  The cone $C(X,L)$ has rational singularities if and only if $H^p(X,L^m) = 0$ for all $p >0$ and $m\ge 0$.  
\end{theorem}

%\begin{example}
%While all normal quotients of rational singularities are rational by \cite[Proposition 5.13]{KollarMori}, one can use this criterion to give an example of a finite quotient $Z_1 \to Z_2$ such that $Z_2$ has rational singularities but $Z_1$ does not.  Let $X$ be a Godeaux surface, a surface of general type with $q = p_g = 0$.  We compute $H^2(X, \omega_X) = H^0(X, O_X) = \CC \ne 0$.  By Theorem \ref{thm:rationalcones}, $C(X,\omega_X)$ does not have rational singularities.  However, by the Kodaira Vanishing Theorem, $H^i(X, \omega_X^{\otimes m}) = 0$ for $m\ge 2$ and $i >0$, so the cone $C(X, \omega_X^{\otimes 2})$ has rational singularities.  There is a finite map $C(X, \omega_X) \to C(X, \omega_X^{\otimes 2})$, and gives the desired example.  
%\end{example}

\noindent The following generalization classifies intermediate rationalizations of singularities.

%We generalize the criterion for rationality in Theorem \ref{thm:rationalcones} to partial resolutions of $C_p(X,L)$ in terms of cohomologically trivial fibrations $X \to Y$.  When discussing cones, we will require a slightly stronger version of a rationalization of singularities. 

%If $X$ is smooth and projective and $L$ is an ample line bundle on $X$ the cone $C_p(X,L)$ has a canonical resolution
%\[
%\pi\cl \PP \ra C_p(X,L)
%\]
%by $\PP = \PP(\Oc\oplus L)$. 

%\begin{definition}
%We say a variety $R$ is a \textbf{rationalization of singularities of $C_p(X,L)$} if $R$ has rational singularities and fits into a diagram:
%\[
%\begin{tikzcd}
%\PP\arrow[rr,"\pi"]\arrow[dr]&&C_p(X,L).\\
%&R\arrow[ur]&
%\end{tikzcd}
%\]
%\end{definition}

%We have the following characterization of rationalizations of singularities of $C_p(X,L)$.

\begin{theorem}\label{thm:rationalizations}
Let $X$ be a smooth projective variety with an ample line bundle $L$.  There is a bijective correspondence
\[
\left\{ \begin{array}{l}
\text{int. rationalizations}\\
\text{of sings. of $C(X,L)$}\\
\end{array} \right\}\longleftrightarrow \left\{\begin{array}{l}
\text{regular and cohom. trivial fibrations $f\cl X\ra Y$, such that}\\
\text{$Y$ has rational sings. and $R^pf_*(L^m) = 0$ for $p,m > 0$}  
\end{array}  \right\}.
\]
\end{theorem}

\begin{remark}\label{KnefOK}
If $L$ is sufficiently positive (e.g. if $\omega^{-1}\otimes L$ is also ample) then the vanishing of $R^pf_*(L^m)$ for $p, m>0$ is automatic. Thus Theorem~\ref{thm:rationalizations} implies Theorem~\ref{intratsings}. Note that $L$ is always ``sufficiently positive" if $-K_X$ is nef.
\end{remark}

\begin{remark}
If $\hpo(X)= 0$ for all $p > 0$, then Theorem~\ref{thm:rationalizations} implies Theorem~\ref{thm:rationalcones} (at least in the case $X$ is smooth).
\end{remark}

\begin{proof}
By \cite[Thm. 7.1]{KollarHigherDirect} (or see Theorem~\ref{thm:kollar} and Corollary~\ref{CohTrivConst}), given a smooth projective variety $X$ and a regular fibration  $f: X \to Y$ with $Y$ normal, the following are equivalent: 
\begin{enumerate}
    \item $R^pf_*\Oc_X = 0$ for $p >0$ 
    \item $Y$ has rational singularities and $f$ is a cohomologically trivial fibration.
\end{enumerate}
Thus the conditions on the right hand side of the theorem can be rephrased as regular fibrations $f\cl X\ra Y$ such that $R^pf_*L^m=0$ for all $p>0$ and all $m\ge 0$.

Start with an intermediate rationalization of singularities:
\[
\begin{tikzcd}
\PP(\Oc\oplus L)\arrow[r]\arrow[rr,bend right=40,"\mu"]\arrow[r,"h"]&\Rb\arrow[r,"g"]&C(X,L)
\end{tikzcd}
\]
Note that the exceptional divisor $E$ of $\mu$ is isomorphic to $X$. Define $Y:=h(E)$ to be the image of $E$ in $\Rb$. We call the induced map $f\cl X\ra Y$. We want to show that $R^pf_*L^m=0$ for all $p>0$ and $m\ge 0$. Note that the thickening $mE$ admits a map to $X$ (the projection $mE\subset \PP(\Oc\oplus L)\xrightarrow{\pi}X$) which makes $\Oc_{mE}$ into a graded $\Oc_X$-algebra, and we may write
\[
\Oc_{mE}=\Oc_X \oplus L \oplus\cdots \oplus L^{m-1},
\]
as a graded $\Oc_X$-module. By the theorem on formal functions,
\[
\widehat{R^ph_*\left(\Oc_{\PP(\Oc\oplus L)}\right)_Y} = \oplus_{m\ge 0} R^p f_* (L^m).
\]
The assumption that $R$ has rational singularities implies the left hand side vanishes. Therefore $R^pf_*L^m=0$ for all $p> 0$ and all $m\ge 0$.

In the other direction, start with a cohomologically trivial fibration such that $R^pf_*L^m=0$ for all $p>0$ and $m\ge 0$. We need to construct an intermediate rationalization. Let $\pi\cl \PP(\Oc\oplus L)\ra X$ denote the projection onto $X$. Define $\Rb=\Rb_{X,L,f}$ to be the normalization of the image of the map
\[
\phi=(\mu,f\circ \pi)\cl \PP(\Oc\oplus L)\ra C(X,L)\times Y.
\]
Let $h\cl \PP(\Oc\oplus L)\ra R$ and $g\cl \Rb \ra C(X,L)$ denote the induced maps. Clearly, the sheaf $R^ph_*(\Oc_{\PP(\Oc\oplus L)})$ is supported on a thickening of $h(E)\subset \Rb$. By applying the theorem on formal functions in the same way as above, we get that $\Rb$ has rational singularities and thus defines an intermediate rationalization of singularities of $C(X,L)$. Showing that these constructions are compatible is straightforward.
\end{proof}

\begin{example}
We give an example of a cone with infinitely many intermediate rationalizations. Consider a K3 surface $X$ with infinitely many (-2)-curves. Each curve $C$ is contractible and the contraction defines a map $f_C: X \to Y$, where $Y$ has a single canonical (and thus rational) singularity. As $K_X=0$ is nef, Remark~\ref{KnefOK} implies that any ample line bundle $L$ on $X$ is ``sufficiently positive" in the sense of Theorem~\ref{intratsings}. Thus by Theorem~\ref{intratsings}, there are infinitely many (non $\QQ$-factorial) intermediate rationalizations of singularities of $C(X,L)$. And in fact, there cannot exist a ``minimal" one. (It is maybe worth noting that although there are infinitely many (-2)-curves, by \cite[Thm. 0.1(b)]{Sterk} these (-2)-curves have only finitely many orbits under the automorphism group $\mathrm{Aut}(X)$.)
\end{example}

\section{Maximal Chow constant and Chow trivial fibrations}\label{sec:maxchowconst}

In this section we show that maximal Chow constant fibrations and maximal Chow trivial fibrations exist. One of the key points is that Chow constant fibrations are the fibrations whose fibers are Chow constant subvarieties (see Theorem~\ref{thm:chowconstantfibers}). The existence of maximal Chow constant fibrations is in some sense due to \Roi tman \cite[Lemma 2]{RoitmanMCC}. The construction is quite general and seemingly well known to experts and it is possible there is a more original reference. We start by recalling \Roi tman's construction. Moreover, we give criteria for the nontriviality of these maximal fibrations.

Let $X$ be a smooth complex projective variety. Let $\Chow(X)$ denote the Chow variety which parameterizes cycles in $X$. Let $W\subset X\times X$ be an equivalence relation which is a countable union of closed irreducible subsets $W=\cup_{i\in \NN} W_i$ (assume no factors are repeated). \Roi tman constructs a maximal quotient $\eta\cl X\dra X/W$ with $W$-equivalent fibers.

\begin{proposition}[{\cite[Lem. 2]{RoitmanMCC}}]\label{prop:maximalWfibration}
Let $X$ and $W$ be as above.
\begin{enumerate}
\item There is a unique maximal and irreducible component $W_0\subset W$ which contains the diagonal $\Delta_X\subset X\times X$.
\item $W_0$ induces a rational map:
\[
\eta\cl X\dra \Chow(X),
\]
to the Chow variety of $X$, and a general fiber of $\eta$ is irreducible. (Thus if we define $Y$ to be a resolution of singularities of the closure of the image of $\eta$ then
\[
\eta\cl X\dra Y
\]
is a fibration.)
\item The fibration $\eta$ is uniquely maximal in the following sense, if $\phi\cl X \dra Z$ is another fibration then the fibers of $\phi$ are equivalent under the relation $W\iff\eta$ factors through $\phi$.
\end{enumerate}
\end{proposition}

\begin{definition}
We call the map $\eta\cl X\dra Y$ in the previous proposition the \textbf{maximal W-constant fibration}. When $W\subset X\times X$ is the equivalence relation defined by equivalence of points in $\CH_0(X)$ we say $Y$ is the \textbf{maximal Chow constant fibration}. Thus Proposition~\ref{prop:maximalWfibration} implies Theorem~\ref{MaxConstFibs} for Chow constant fibrations.
\end{definition}

\begin{proof}[Sketch of Proof] Throughout, for a subvariety $W'\subset X\times X$ we use
\[
W'_z:=W'\cap (z\times X)\subset X
\]
to denote the fiber of $W'$ over $z\in X$ under the first projection. First, we remark that if $z\in X$ is general and $W'\subset X\times X$ is irreducible and contains the diagonal then every component of $W'_z$ contains the diagonal point $(z,z)\in \Delta_X$.

To prove (1), assume that there are two maximal components $W_0, W_1\subset W$ which contain the diagonal. The idea is to use the transitivity of $W$. We make the following assertion, which is a standard application of the Baire category theorem:
\begin{center}
$(\star)$ Maximality of $W_0$ along with the uncountability of $\CC$ guarantees\\
that a very general point $(x_1,x_2)\in W_0$ satisfies $(x_1,x_2)\not\in \bigcup\limits_{i\in \NN i\ne 0} W_i.$
\end{center}
\noindent Let $z\in X$ be very general and let $(z,x)\in W_0$ be a very general point in $(W_0)_z$. By transitivity, $z\times (W_1)_x\subset W$ and contains the very general point $(z,x)\in z\times W_z$. It follows from $(\star)$ that $(W_1)_x\subset (W_0)_z$. Taking the limit as $x$ approaches $z$ shows $(W_1)_z\subset (W_0)_z.$ (This uses that $z$ is very general, so the projection of $W_1$ onto the first factor is flat in a neighborhood of $z$.) As $z$ is very general and $W_0$ and $W_1$ are maximal, we have $W_0=W_1$.

So let $W_0\subset W$ be the unique, maximal irreducible component which contains $\Delta_X$. Clearly $W_0\subset X\times X$ is a reflexive subset. A similar argument to the previous paragraph implies that for $z$ general $(W_0)_z$ is irreducible, and it also shows that if $(z,x)\in (W_0)_z$ is general, then $(W_0)_x= (W_0)_z\subset X$. Thus we have shown (2) and define the maximal $W$-constant fibration to be the map
\begin{center}
$\eta\cl X \dra \Chow(X)$\\
sending $x\mapsto \eta(x):=[(W_0)_x].$
\end{center}
By reflexivity, for a general point $x\in X$, the closure of the fibers of $\eta$ at $x$ is $(W_0)_x$.

The universal property (3) follows from the fact that pairs of points in a general fiber of $\phi$ gives rise to an irreducible component of $W$ which contains $\Delta_X$. Unique maximality of $W_0$ then implies that $\eta$ factors through $\phi$.
\end{proof}

\begin{theorem}
Let $\eta\cl X\dra Y$ be the maximal Chow constant fibration of a smooth $n$-dimensional projective variety $X$. The following are equivalent.
\begin{enumerate}
\item $\dim(Y)\le d$.
\item $\CH_0(X)$ is supported on a variety of dimension $d$.
\item $\CH_0(X)$ is supported on a smooth irreducible variety of dimension $d$.
\item For every point $x\in X$ there is a dimension $n-d$ subvariety $V\subset X$ such that every point $x'\in V$ satisfies $x=x'\in \CH_0(X)$.
\end{enumerate}
\end{theorem}

\begin{proof}
For (1)$\implies$ (4) let $\Gamma_\eta\subset X\times Y$ be the closure of the graph of $\eta$. Let $y$ be a point in the image of $(\Gamma_\eta)_x$. Then $(\Gamma_\eta)_y\subset X$ consists of points rationally equivalent to $x\in X$ and has dimension at least $n-d$. For (4)$\implies$(3), take a general complete intersection of ample divisors on $X$. (3)$\implies$(2) is clear.

What remains is (2)$\implies$(1). Suppose that $\CH_0(X)$ is supported on a union of subvarieties $V\subset X$ such that $\dim(V)=d$. Let $W\subset X\times X$ correspond to equivalence in $\CH_0(X)$. Take the preimage of $W_V$ in $V\times X$, i.e. define
\[
W_V:=\{(v,x)\in V\times X \mid v=x\in \CH_0(X)\}.
\]
$W_V$ is a countable union of subvarieties, and as $\CH_0(X)$ is supported on $V$, we have that there is a component of $W_1\subset W_V$ such that the projection
\[
p_2\cl V\times X \ra X
\]
maps $W_1$ surjectively onto $X$. For a point $v\in V$, $p_2((W_1)_v)$ consists of points which are rationally equivalent. Surjectivity of $p_2|_{W_1}$ implies that through a general point $x\in X$ there is a Chow constant subvariety $Z_x\subset X$ containing $x$ and satisfying $\dim(Z_x)\ge n-d$. Therefore, the maximal component $W_0\subset W$ containing the diagonal has fiber dimension $\dim((W_0)_x)\ge n-d$. Thus by construction of $Y$, $\dim(Y)\le d$.
\end{proof}

To obtain the maximal Chow trivial fibration is not much more difficult. It will be necessary to construct a relative Chow constant fibration. Let $\pi\cl X\ra Z$ be a regular fibration of projective varieties. Assume that $X$ is smooth. Consider the equivalence relation $W(\pi)\subset X\times X$ defined by:
\[
W(\pi):=\{ (x_1,x_2)\in X|\pi(x_1)=\pi(x_2)=z\text{ and }x_1=x_2 \in \CH_0(X_z)\}.
\]
Then we have $W(\pi)=\bigcup\limits_{i\in\NN} W(\pi)_i$ is a countable union of closed subsets. Let $W(\pi)_0$ be the unique maximal component containing $\Delta_X$. 

\begin{lemma}\label{lem:relativemaxWfibration}
\begin{enumerate}
\item Let $Y:=(X/W(\pi))$ be the maximal $W(\pi)$-constant fibration. There is a commutative diagram:
\[
\begin{tikzcd}
X\arrow[dr,swap,"\pi"]\arrow[rr,dashed,"\eta_{\pi}"]&&Y:=(X/W(\pi)).\arrow[dl,dashed]\\
&Z&
\end{tikzcd}
\]
\item If $z\in Z$ is very general, then $(W(\pi)_0)_z=(W(\pi)_z)_0$, i.e. for very general $z\in Z$ the map
\[
\eta_\pi|_{X_z}\cl X_z\dra Y_z
\]
is equivalent to the maximal Chow constant fibration of $X_z$.
\end{enumerate}
\end{lemma}

\begin{proof}
(1) holds because the closure of the fibers of $\eta_x$ are contained in fibers of $\pi$. (2) follows from the assertion $(\star)$ in the sketch of the proof of Proposition \ref{prop:maximalWfibration}.
\end{proof}

To construct the maximal Chow-trivial fibration of X we consider the following sequence:
\[
\begin{tikzcd}
&X=X_0\arrow[dl,"\pi_0"]\arrow[d,dashed,"\eta_0"]&X_1\arrow[dl,"\pi_1"]\arrow[d,dashed,"\eta_1"]\arrow[l,swap,"\psi_1"]&\cdots\arrow[l,swap,"\psi_2"]\arrow[dl,"\pi_2"]&X_n\arrow[dl,swap,"\pi_n"]\arrow[d,dashed,"\eta_n"]\arrow[l,swap,"\psi_n"]&\cdots\arrow[dl,"\pi_{n+1}"]\arrow[l,swap,"\psi_{n+1}"]\\
\Spec(\CC)&Y_0\arrow[l]&Y_1\arrow[l,dashed]&\cdots\arrow[l,dashed]&Y_n\arrow[l,dashed]&\cdots\arrow[l,dashed]
\end{tikzcd}
\]
Each $X_i$ is a resolution of the map $\eta_{i-1}$, the map $\psi_i$ is birational, and $\eta_i$ is defined to be the maximal relative Chow constant fibration of $\pi_i$. The following proposition implies Theorem~\ref{MaxTrivFibs} for Chow trivial fibrations.

\begin{proposition}
For $n\gg 0$, we have $Y_n\simeq_\bir Y_{n+1}\simeq_\bir\cdots$. Set $Y_\infty:=Y_n$.
\begin{enumerate}
\item The composition
\[
\begin{tikzcd}
X\arrow[r,dashed,swap,"\simeq_\bir"]\arrow[rr,bend left=30,dashed,"\eta_\infty"]&X_{n+1}\arrow[r,swap,"\pi_{n+1}"]&Y_\infty
\end{tikzcd}
\]
is a Chow-trivial fibration.
\item If $\phi\cl X\dra Z$ is another Chow-trivial fibration, then $\eta_\infty$ factors through $\phi$.
\item We have $\dim(Y_\infty)\le m$ if and only if through a very general point $x\in X$, there is a Chow trivial subvariety $x\in V$ of codimension $\ge m$.
\end{enumerate}
\end{proposition}

\begin{proof}
(1) follows from Lemma \ref{lem:relativemaxWfibration}(2) and the fact that the map from a fiber $X_y$ to a point is a Chow constant fibration $\iff\CH_0(X_y)\cong \ZZ$. (2) and (3) can be checked for $\eta_1,\cdots, \eta_n$.
\end{proof}

\begin{definition}
Let $\eta_{\infty}$ and $Y_{\infty}$ be as in the previous proposition.  The \textbf{maximal Chow trivial fibration} is the rational map
\[
\eta_\infty \cl X\dra Y_\infty.
\]
\end{definition}

\begin{remark}\label{genchowprop}
It follows from Lemma~\ref{ruledness} and Corollary~\ref{cor:compositionofchowfib} that the maximal Chow constant fibration and the maximal Chow trivial are almost holomorphic (see Def.~\ref{almhom}). As a consequence if $x\in X$ is very general, any Chow constant (resp. trivial) subvariety is contained in a smooth Chow constant (resp. trivial) subvariety.
\end{remark}

\section{Maximal Cohomologically Constant and Trivial Fibrations}\label{sec:maxcohomological}

The aim of this section is to prove the existence of maximal cohomologically constant and trivial fibrations. In fact, we show that given any integrable distribution $\Dc$ on a smooth complex projective variety (e.g. Voisin's distribution, Def.~\ref{def:VD}) there is a maximal fibration whose generic fibers are contained in the leaves of the associated foliation. The idea of using a foliation to prove the existence of maximal fibrations was suggested to us by Claire Voisin.

\begin{definition}
Let $\Dc\subset T_X$ be an integrable distribution on a smooth variety $X$. We say that a subvariety $V\subset X$ is \textbf{contained in $\Dc$} if at a general point $x\in V$,
\begin{enumerate}
\item $\Dc$ is locally a vector subbundle of $T_X$ at $x$ (i.e. the quotient $T_X/\Dc$ is locally free at $x$), and
\item the subspace $T_V|_x\subset T_X|_x$ is contained in $\Dc|_x$.
\end{enumerate}
\end{definition}

\begin{remark}
Assume that $V$ intersects the open set $U\subset X$ where $\Dc\subset T_X$ is a sub-vector bundle. Consider the composition
\[
\begin{tikzcd}
T_V\arrow[r]\arrow[rr,bend right=30,"\alpha"]&T_X|_V\arrow[r]&(T_X|_V)/(\Dc|_V)
\end{tikzcd}
\]
Then $V$ is contained in $\Dc\iff \alpha|_{V\cap U}\equiv 0$.
\end{remark}

\begin{remark}\label{leafy}
If $U$ is the open set where $\Dc\subset T_X$ is a subbundle, then  $\Dc$ gives rise to a foliation on $U$. Assuming $x\in V\cap U$, then $V$ is contained in $\Dc \iff$ analytically locally around $x$, $V$ is contained in a leaf of the foliation.
\end{remark}

\begin{definition}
Let $X$ be a smooth projective variety.  A fibration $f: X \dashrightarrow Y$ is a \textbf{$\Dc$-constant fibration} if the general fiber is contained in $\Dc$.  
\end{definition}

\begin{remark}
By Proposition~\ref{descendingForms} we have that for a smooth projective variety $X$, a fibration $f\cl X\dra Y$ is Chow constant $\iff$ it is $\Vc_X$-constant $\iff$ a general fiber is contained in Voisin's distribution.
\end{remark}

To construct maximal $\Dc$-constant fibrations, we want to show that there is a maximal family of $\Dc$-constant subvarieties. We proceed as follows.  Let $\Hilb(X)$ be the Hilbert scheme of $X$ and consider the locally closed subset
\[
\Dc\Var:=\left\{ \hspace{.05in} [V]\in \Hilb(X) \hspace{.05in} \middle| \begin{array}{l} V\text{ is a variety, and}\\ V \text{ is contained in }\Dc\end{array}\right\}\subset \Hilb(X),
\]
with the reduced scheme structure. Then $\Dc\Var$ is a countable union of quasiprojective varieties. Write
\[
\Dc\Var=\bigcup\limits_{i\in \NN} S_i
\]
where each $S_i$ is a subvariety and $S_i\subset S_j\iff i=j$. Let $\Sbar_i$ denote the closure of $S_i\subset \Hilb(X)$. Write $\Fbar_i$ for the universal family over $\Sbar_i$. $\Fbar_i$ comes equipped with projections:
\[
\begin{tikzcd}
\Fbar_i\arrow[r,"q_i"]\arrow[d,"p_i"]&X.\\
\Sbar_i
\end{tikzcd}
\]
It is natural to restrict ourselves to the varieties contained in $\Dc$ which sweep out $X$. Define
\begin{center}
$I:=\{i\in \NN \mid q_i\text{ is dominant}\}\subset \NN,$ and $\DDom := \{\Sbar_i\}_{i\in I}$
\end{center}

\begin{remark}\label{vgenpoint}
Let $x\in X$ be a very general point, and let $V\subset X$ be a subvariety contained in $\Dc$. If $x\in V$ then there exists $\Sbar_i\in \DDom$ such that $[V]\in \Sbar_i$.
\end{remark}

We make $\DDom$ into a partially ordered set by 
\[
\Sbar_i\le \Sbar_j \iff\text{ for }[V]\in \Sbar_i \text{ general, } \exists\text{ }[W]\in \Sbar_j \text{ such that }V\subset W.
\]
For any $i,j \in \NN$ we define
\[
\Sbar_{\ge i}:=\{ \Sbar_k| S_k\ge S_i\} \subset \DDom,\text{ and } \Sbar_i \vee \Sbar_j := \Sbar_{\ge i} \cap \Sbar_{\ge j} .
\]

\begin{construction}
We show that for any $\Sbar_i, \Sbar_j\in \DDom$, the set $\Sbar_i\vee \Sbar_j\ne \emptyset.$ We may choose very general points $x\in X$,  $[V_1]\in \Sbar_i$, and $[V_2]\in \Sbar_j$ subject to the following conditions:
\end{construction}

\begin{enumerate}[label=(\roman*)]
\item $x$ is very general in the sense of Remark~\ref{vgenpoint},
\item for any $\Sbar_k$ and $[V_3]\in \Sbar_k$ such that $V_1, V_2\subset V_3$ we have $\Sbar_k\in \Sbar_i\vee \Sbar_j$,
\item $x\in V_1$ and $x\in V_2$.
\end{enumerate}

Let $Q\subset q_j^{-1}(V_1)$ be an irreducible component such that $[V_2]\in p_j(Q)$ and define:
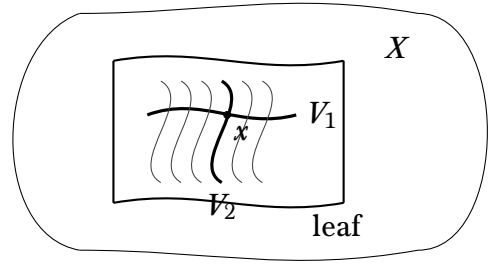
\begin{wrapfigure}{r}{-20pt}
\centering
\hspace{20pt}
\begin{tikzpicture}[scale=.9]
 \draw[very thick] (0,0) to [out=20,in=195] (2.2,0);
 \draw[color=darkgray] (.2,.5) to [out=-30,in=150] (.2,-1);
 \draw[color=darkgray] (.5,.5) to [out=-30,in=150] (.5,-1);
 \draw[color=darkgray] (.8,.5) to [out=-30,in=150] (.8,-1);
 \draw[very thick] (1.1,.5) to [out=-30,in=150] (1.1,-1);
 \draw[color=darkgray] (1.4,.5) to [out=-30,in=150] (1.4,-1);
 \draw[color=darkgray] (1.7,.5) to [out=-30,in=150] (1.7,-1);
 \node[right] at (2.2,0) {$V_1$};
 \draw[fill] (1.18,0) circle [radius=0.05];
 \node[below right] at (1.1,0) {$x$};
 \node[below] at (1.1,-1) {$V_2$};
 \draw[thick] (-.5,.8) to [out=10,in=190] (2.9,.8);
 \draw[thick] (-.5,-1.3) to [out=10,in=190] (2.9,-1.3);
 \draw[thick] (-.5,.8) to (-.5,-1.3);
 \draw[thick] (2.9,.8) to (2.9,-1.3);
 \draw (-1,1.5) to [out=0,in=170] (4,1.5);
 \draw (-1,-2) to [out=0,in=190] (4,-2);
 \draw (-1,1.5) to [out=200,in=170] (-1,-2);
 \draw (4,1.5) to [out=0,in=0] (4,-2);
 \node at (3.7,1) {$X$};
 \node[below] at (2.8,-1.3) {leaf};
\end{tikzpicture}
\caption{The deformations of $V_2$ considered are locally contained in a leaf. As they meet $V_1$, which is also contained in a leaf, they are all contained in the same leaf.}
\end{wrapfigure}
\hspace{1in}$V_3:=q_j(p_j^{-1}(p_j(Q))).$

\noindent Then, $V_3\subset X$ is a subvariety containing both $V_1$ and $V_2$. If we can show that $V_3$ is contained in $\Dc$, then by condition (iv) above we are done. As $x$ is very general, $\Dc$ is a subbundle of $T_X$ in a neighborhood of $x$. Therefore, $V_3$ is (the closure of) a union of $\Dc$-constant subvarieties which are deformations of $V_2$, all of which meet $V_1$. As $V_1$ is also $\Dc$-constant, in an analytic neighborhood of $x$ every deformation of $V_2$ must be contained in the leaf which contains $x$. Therefore, $V_3$ is analytically locally contained in the leaf at $x$, hence by Remark~\ref{leafy} we see that $V_3$ is contained in $\Dc$.

\begin{remark}\label{rem:joinconstruction}
Note that the construction of $V_3$ involves choices and is asymmetric in $i$ and $j$.  The following properties hold:
\begin{enumerate}
\item if $V_1=V_3$ then $p_j^{-1}(p_j(Q))=Q$, and
\item if $V_2=V_3$ then $p_j(Q)$ is a single point and the map $q_j\cl\Fbar_j\ra X$ is generically finite.
\end{enumerate}
\end{remark}

As the dimension of subvarieties of $X$ are bounded from above, there is a unique maximal family $\Sbar_0\in\DDom$. Thus $\Sbar_0\vee \Sbar_0=\{ \Sbar_0\}$. 

\begin{theorem}\label{thm:maxDconstantfib} With the above setup.
\begin{enumerate}
\item The map $q_0\cl \Fbar_0\ra X$ is birational, and the composition
\[
\begin{tikzcd}
X\arrow[rr,dashed,bend right=40,"\eta"]\arrow[r,dashed,"q_0^{-1}"]&\Fbar_0\arrow[r,"p_0"]&\Sbar_0
\end{tikzcd}
\]
is a $\Dc$-constant fibration.
\item If $V\subset X$ is a $\Dc$-constant subvariety which contains a very general point $x\in X$, then $V$ is contracted by $\eta$.
\item If $\phi\cl X\dra Y$ is any $\Dc$-constant fibration, then $\eta$ factors through $\phi$.
\end{enumerate}
\end{theorem}

\begin{proof}
It follows from Remark \ref{rem:joinconstruction} that $q_0$ is birational. By construction, the map $\eta$ is a $\Dc$-constant fibration, which proves (1). To prove (2), by Remark \ref{vgenpoint} there exists an $\Sbar_i \in \DDom$ such that $[V] \in \Sbar_i$. But $\Sbar_0$ is uniquely maximal, so $V$ must be contained in a fiber of $\eta.$ Finally, to prove (3), let $V_y = \phi^{-1}(y)$ be the closure of a very general fiber of $\phi$.  By (2), this must be contracted by $\eta$. As a consequence, the closure of the image of $X\dra Y\times \Sbar_0$ is the graph of the appropriate rational map $Y\dra \Sbar_0$.
\end{proof}

\begin{definition}
After resolving the singularities of $\Sbar_0$, we call the map $\eta$ the \textbf{maximal $\Dc$-constant fibration}.  When $\Dc = \Vc_X$, $\eta$ is the \textbf{maximal cohomologically constant fibration}.
\end{definition}

\noindent This proves Theorem~\ref{MaxConstFibs} for cohomologically constant fibrations. Furthermore, we have

\begin{corollary}
Given a regular fibration $\pi\cl X\ra Z$ of smooth projective varieties, there is a maximal relative cohomologically constant fibration
\[
\begin{tikzcd}
X\arrow[rr,dashed,"\eta_\pi"]\arrow[dr,swap,"\pi"]&&Y\arrow[dl]\\
&Z&
\end{tikzcd}
\]
\end{corollary}

\begin{proof}
Apply the construction of the maximal $\Dc$-constant fibration when $\Dc=\Vc_\pi$, the relative Voisin distribution (Remark \ref{rmk:relvoisindist}). Then, if necessary, resolve the map $Y\ra Z$. This defines a maximal relative cohomologically constant fibration $\eta_\pi \cl X \dra Y$ over $Z$, as desired.  
\end{proof}

\begin{lemma}\label{lem:maxcohfibrationonfibers}
Given a regular fibration $\pi \cl X \ra Z$ of smooth projective varieties, for a very general point $z \in Z$, the maximal relative cohomologically constant fibration $\eta_\pi$ induces the maximal cohomologically constant fibration on the fibers over $z$ (i.e., for general $z \in Z$, the map $\eta_\pi \vert_{X_z} \cl X_z \dra Y_z$ is the maximal cohomologically constant fibration of $X_z$).
\end{lemma}

\begin{proof}
As stated in Remark~\ref{rmk:relvoisindist}, for a general point $x \in X_z$, the distribution $\Vc_{\pi}$ is equal to $\Vc_{X_z}$.  The statement follows.
\end{proof}

To obtain the maximal cohomologically trivial fibration of a smooth variety $X$, we proceed as in the end of \S\ref{sec:maxchowconst}. Consider the sequence
\[
\begin{tikzcd}
&X=X_0\arrow[dl,"\pi_0"]\arrow[d,dashed,"\eta_0"]&X_1\arrow[dl,"\pi_1"]\arrow[d,dashed,"\eta_1"]\arrow[l,swap,"\psi_1"]&\cdots\arrow[l,swap,"\psi_2"]\arrow[dl,"\pi_2"]&X_n\arrow[dl,swap,"\pi_n"]\arrow[d,dashed,"\eta_n"]\arrow[l,swap,"\psi_n"]&\cdots\arrow[dl,"\pi_{n+1}"]\arrow[l,swap,"\psi_{n+1}"]\\
\Spec(\CC)&Y_0\arrow[l]&Y_1\arrow[l]&\cdots\arrow[l]&Y_n\arrow[l]&\cdots\arrow[l]
\end{tikzcd}
\]
where the birational map $\psi_i \cl X_i \ra X_{i-1}$ is a resolution of the map $\eta_{i-1}$ and $\eta_i$ is defined to be the maximal relative cohomologically constant fibration of $\pi_i$.

For $n \gg 0$, $Y_n \simeq_\bir Y_{n+1} \simeq_\bir \dots$ (as the fiber dimension is bounded). For $n$ sufficiently large define $Y_{\infty} := Y_n$. Define $\eta_\infty \cl X \dra Y_{\infty}$ to be the composition $\eta_{\infty} \cl X \simeq_\bir X_{n+1} \ra Y_n = Y_{\infty}$.  This is the maximal cohomologically trivial fibration of $X$.  

\begin{proposition}
\begin{enumerate}
\item The rational map $\eta_{\infty} \cl X \dra Y_{\infty}$ is a cohomologically trivial fibration.
\item If $\phi\cl X\dra Z$ is another cohomologically trivial fibration, then $\eta_\infty$ factors through $\phi$.
\end{enumerate}
\end{proposition}

\begin{proof}
To show (1), as $Y_n\simeq_\bir Y_{n+1}\simeq_\bir\cdots$ it follows by Lemma~\ref{lem:maxcohfibrationonfibers} that for a very general fiber $X_y$ of $\eta_n$, the map from $X_y$ to a point is a cohomologically constant fibration. Thus $X_y$ is cohomologically trivial. (2) can be checked for each map $\eta_i$ using Theorem \ref{thm:maxDconstantfib}(2).
\end{proof}

This proves Theorem~\ref{MaxTrivFibs} for cohomologically trivial fibrations.

\begin{definition}
For a smooth projective variety $X$, the rational map $\eta_\infty \cl X \dra Y_\infty$ defined above is the \textbf{maximal cohomologically trivial fibration}. 
\end{definition}

\begin{remark}
As in Remark~\ref{genchowprop}, it follows from Lemma~\ref{ruledness} and Corollary~\ref{cor:compositionofcohfib} that the maximal cohomologically constant fibration and the maximal cohomologically trivial fibration are almost holomorphic.
\end{remark}

\appendix

\section{}\label{appendix}

Throughout this appendix, by a \textbf{curve} we mean a reduced, 1-dimensional scheme of finite type over an arbitrary field $k$. The point of this appendix is to prove that if $X$ is a projective variety over an arbitrary field $k$ and $\CH_0(X)$ (resp. $\CH_0(X)\otimes \QQ$) is supported on a curve, then $\CH_0(X)$ (resp. $\CH_0(X)\otimes \QQ$) is finite dimensional in the sense of Mumford (see Def.~\ref{MumFinite}). The main technical problems arise in considering reducible and singular curves. If
\[
C=C_1\sqcup C_2
\]
is a disjoint union of two projective curves then $\CH_0(C)$ is \textit{not} finite dimensional, as
\[
\CH_0(C)_0=\bigcup\limits_{k\in \ZZ}\left(\CH_0(C_1)_k\times \CH_0(C_2)_{-k}\right)
\]
contains divisors with unbounded degree on $C_1$. This issue can be overcome in two parts.
\begin{enumerate}
\item If $C$ is a connected curve, then $\CH_0(C)$ (resp. $\CH_0(C)\otimes \QQ$) is finite dimensional.
\item If $\CH_0(X)$ is supported on a curve, then it is also supported on a connected curve.
\end{enumerate}

\begin{proposition}
If $C=C_1\cup \cdots \cup C_m$ is a projective connected curve, then $\CH_0(C)$ (resp. $\CH_0(C)\otimes \QQ$) is finite dimensional.
\end{proposition}

\begin{proof}
First note that the map of sets:
\[
\CH_0(C)\times \QQ\ra \CH_0(C)\otimes \QQ
\]
is surjective. Thus, if $\CH_0(C)$ is finite dimensional, then so is $\CH_0(C)\otimes \QQ$.

Let
\[
\nu\cl D\ra C
\]
be the normalization of $C$, i.e. $D=D_1\sqcup \cdots \sqcup D_m$ where $D_i$ is the normalization of $C_i$. Let $U\subset C$ be the regular locus. Then $U$ is nonempty in each component $C_i$, and the map $D\ra C$ is an isomorphism over $U$. Let $a$ (resp. $b$) be the number of closed points in $D\setminus U$ (resp. $C\setminus U$). Consider the localization sequences:
\[
\begin{tikzcd}
\ZZ^a\arrow[r]\arrow[d,"\phi"]&\CH_0(D)\arrow[r]\arrow[d,"\nu_*"]&\CH_0(U)\arrow[d,"\cong"]\arrow[r]&0\\
\ZZ^b\arrow[r]&\CH_0(C)\arrow[r]&\CH_0(U)\arrow[r]&0.
\end{tikzcd}
\]
The rows of the above diagram are exact. Moreover, the image of $\phi$ is full rank in $\ZZ^b$ (every singularity point in $C$ has a nonempty preimage in $D$). As a consequence, the cokernel $\CH_0(C)/\mathrm{Im}(\nu_*)$ is finite, thus the cokernel doesn't affect the finite dimensionality of $\CH_0(C)$.

So, to prove $\CH_0(C)$ is finite dimensional is suffices to show there exists an integer $d>0$ such that the difference map:
\[
\sym^d(D)(k)\times \sym^d(D)(k)\ra\CH_0(C)
\]
contains $\mathrm{Im}(\nu_*)\cap \CH_0(C)_0$. Consider the following commutative diagram of degrees:
\[
\begin{tikzcd}
0\arrow[r]&\ker(\nu_*)\arrow[r]\arrow[d]&\CH_0(D)\arrow[r,"\nu_*"]\arrow[d,"\underline{\deg}"]&\CH_0(C)\arrow[d,"\deg"]&\\
0\arrow[r]&\ZZ^{m-1}\arrow[r]&\ZZ^m\arrow[r,"\sum"]&\ZZ\arrow[r]&0
\end{tikzcd}
\]
where $\underline{\deg} = (\deg_1,\cdots, \deg_m)$ and $\deg_i$ is the degree map on $D_i$. It is easy to see that the image of $\underline{\deg}$ has full rank, and it follows that the image of $\ker(\nu_*)$ in $\ZZ^{m-1}$ has full rank. Let
\[
\Delta_N:=\{\beta\in \CH_0(D)_0 \mid|\deg_i(\beta)|\le N\text{ }\forall i\}\subset \CH_0(D)_0,
\]
i.e. $\Delta_N$ is the subset of total degree 0 cycles such that each $|\deg_i|$ is bounded by $N$. As $\ker(\nu_*)$ has full rank in $\ZZ^{m-1}$, it follows that there exists $N\ge 0$ such that any 
\[
\mathrm{Im}(\nu_*)\cap \CH_0(C)_0\subset\nu_*(\Delta_N)\subset \CH_0(C).
\]
The proposition now follows from the following lemma.
\end{proof}

Let $D=D_1\sqcup \cdots \sqcup D_m$ and $\Delta_N$ be as in the proof of the proposition.

\begin{lemma}
There exists $d>0$ such that the image of the difference map
\[
\sym^d(D)(k)\times \sym^d(D)(k)\ra \CH_0(D)_0
\]
contains $\Delta_N.$
\end{lemma}

\begin{proof}
Let $\beta\in \Delta_N$ be a degree 0 divisor such that $|\deg_i(\beta)|\le N$ for each $i$. The point is to show that any such $\beta$ is the difference of effective divisors of bounded degree. Let $A\in \CH_0(D)$ be an ample divisor. By applying Riemann-Roch for each component $D_i$, there exists an $\ell\ge 0$ such that $\Oc_D(\ell A + \beta)$ is effective for any $\beta \in \Delta_N$. As a consequence, we can take $d=\ell\cdot \deg(A)$ in the statement.
\end{proof}

Lastly we need the following easy lemma.

\begin{lemma}
Let $X$ be any projective variety over $k$. Any curve $C\subset X$ is contained in a connected curve $C'\subset X$.
\end{lemma}

\begin{proof}
Any two closed points in $X$ can be connected via a connected curve (take a complete intersection of sufficiently ample divisors through both points). The lemma follows.
\end{proof}

Thus we have shown:

\begin{corollary}\label{cor:chow0finitedim}
If $X$ is a variety over an arbitrary field $k$ and $\CH_0(X)$ (resp. $\CH_0(X)\otimes \QQ$) is supported on a curve $C\subset X$, then $\CH_0(X)$ (resp. $\CH_0(X)\otimes \QQ$) is finite dimensional.
\end{corollary}

\bibliographystyle{siam} 
\bibliography{chowquotients} 

\end{document}